\documentclass[AMA,LATO1COL]{WileyNJD-v2}

\articletype{Research article}%

\received{}
\revised{}
\accepted{}

\usepackage{mathrsfs,amsthm}
\usepackage{titlesec}
\usepackage{tikz}
\usepackage{tikz-3dplot}
\usetikzlibrary{calc,hobby}
\usetikzlibrary{decorations.pathmorphing,patterns,decorations,shapes,arrows, intersections,matrix,fit,calc,trees,positioning,arrows,chains, shapes.geometric,shapes,angles,quotes,fit,math}
\usetikzlibrary{arrows,positioning}
\usepackage{float}

\DeclareMathOperator{\trace}{\mathrm{tr}}

\begin{document}

\title{Mathematical analysis of complex SIR model with coinfection and density dependence}

\author[1]{Samia Ghersheen}
\author[1]{Vladimir Kozlov}
\author[1]{Vladimir Tkachev}
\author[2]{Uno Wennergren}

\authormark{Samia Ghersheen, Vladimir Kozlov, Vladimir Tkachev*,  Uno Wennergren}

\address[1]{\orgdiv{Department of Mathematics}, \orgname{Link\"oping University}, \orgaddress{\state{Link\"oping}, \country{Sweden}}}
\address[2]{\orgdiv{Department of Physics, Chemistry, and Biology}, \orgname{Link\"oping University}, \orgaddress{\state{Link\"oping}, \country{Sweden}}}

\corres{*Vladimir Tkachev, This is sample corresponding address. \email{vladimir.tkatjev@liu.se}}


\abstract[Summary]{An SIR model with the coinfection of the two infectious agents in a single host population is considered. The model includes the environmental carry capacity in each class of population. A special case of this model is analyzed and several threshold conditions are obtained which describes the establishment of disease in the population. We prove that for small carrying capacity $K$ there exist a globally stable disease free equilibrium point. Furthermore, we establish the continuity of the transition dynamics of the stable equilibrium point, i.e. we prove that (1) for small values of $K$ there exists a \textit{unique} globally stable equilibrium point, and (b) it moves continuously as $K$ is growing (while its face type may change). This indicate that carrying capacity is the crucial parameter and increase in resources in terms of carrying capacity promotes the risk of infection.}


\maketitle

\section{Introduction}
Coinfection means that a person is affected by more than one infectious agents at a time. Many pathogens that infect humans (e.g., viruses, bacteria, protozoa, fungal parasites) often coexist within individuals \cite{Griffiths}. Consequently co-infection of individual hosts by multiple infectious agents is a phenomenon that is very frequently observed in natural populations. The study of complete dynamics of this phenomenon involves many complexities. By understanding the multiple interactions that cause co-infection, we will be able to understand and intelligently predict how a suite of co-infections will together respond to medical interventions as well as other environmental changes \cite{Viney}.

Mathematical analysis of infectious diseases has significant importance in  infectious disease epidemiology to study not only the dynamics of disease but it is also very helpful to design the practical controlling strategies.
Mathematical analysis and models have successfully explained previously mystifying  observations and played a central part in public health strategies in many countries \cite{May,Glasser}.

Many mathematical studies exist on interaction of multiple strains and multiple disease co-interactions in \cite{Ferguson,Kawaguchi},\cite{Sharp1},  \cite{Chaturvedi}  \cite{Mukandavire}, \cite{Nthiiri,Abu1}, \cite{Okosun}, \cite{ Castillo}. Some of the studies are about the general dynamics of coinfection in  \cite{Zhang}, \cite{ Bichara}, \cite{Martcheva}.

Since coinfection includes a lots of complexities and therefore a lot of different classes and interactions between classes. This implies the necessity
 of advanced mathematical analysis to handle this problem. But there is also a risk that one reaches the boarder of what is actually possible to analyse. Previously, Allen et al. \cite{Allen} studied a SI model with density dependent mortality and coinfection in a single host where one strain is vertically and the other is horizontally transmitted and the model has application on hantavirus and arenavirus and Gao et al. \cite{Gao} studied a SIS model with dual infection. Simultaneous transmission of infection and no immunity has been considered. The study revealed that the coexistence of multiple agents caused co-infection and made the disease dynamics more complicated. It was observed that coexistence of two disease can only occur in the presence of coinfection. In above models they considered that the number of births per unit time is constant.
Our approach in this paper is to show different degrees of complexities and to suggest a model which is an extension of model studied by Ghersheen et al. in \cite{SKTW18a} where an SIR model was analysed. That model describes the coinfection of the two infectious agents in a single host population with an addition of limited growth of susceptible in terms of carrying capacity. Previously, to diminish the complexity, it was considered that there is no interaction between two single infectious agent and coinfection only occur as a result of interaction between coinfected class and single infected class and coinfected class and susceptible class. In this model we add more complexity by relaxing all these assumptions and adding the density dependence in each class.  The addition of two viruses with density dependence for human population  is a new modelling perspective since death and birth rates changes over time in human population. One billion out of eight billions of human population is facing the problem of hunger due to the lack of resources that can fluctuate the birth rates over time. So, in contrast to Allen et al. \cite{Allen} and Gao we first analyse the model with same interaction terms but only the growth of susceptible class is limited in terms of carrying capacity, since infected and recovered population is regulated by its death rate. We assume that infected and recovered individuals cannot reproduce.

  We formulate a SIR model with coinfection and logistic type population growth in each class population to study the effects of carrying capacity on disease dynamics. However, contrary to \cite{Allen}, to study the global behaviour of the system, the reduction of the system is needed to some sense. So in the first place we consider the relative simplified model and only limit the growth of susceptible populations in terms of carrying capacity. We assume that the infected population cannot reproduce due to infection. We carried out the local and  global stability analysis using a generalized Volterra function for each stable point to study the complete dynamics of disease.
 We analyse a SIR model with complete cross immunity.
 In section~\ref{sec:model} we begin with the description of full model and in section 2 we  presented and analysed a submodel. In Section~\ref{sec:general} we recall some general facts about SIR model \eqref{Model}, and in the remained sections we characterize all equilibrium points and give the results regarding local and global stability.

\section{Formulation of the model}\label{sec:model}
The first model is the relevant extension to the model presented in \cite{SKTW18a} to understand the complex dynamics of coinfection. Firstly, we assume that a susceptible individual can be infected with either one or both infectious agents as a result of contact with coinfected individual. Secondly, coinfection occur as a result of contact between two single infected individuals or coinfected and single infected  person. This process is illustrated in the compartmental diagram in \figurename{1}.
\begin{figure}[h]
	\begin{center}
		\begin{tikzpicture}[every node/.style={ minimum height={1cm},minimum width={2cm},thick,align=center}]
		\node[draw] (I1) {$I_1$};
		\node[draw, above right=of I1] (S) {$S$};
		\node[draw, below=of S] (I12) {$I_{12}$};
		\node[draw, below=of I12] (R) {$R$};
		\node[draw, right= of I12] (I2) {$I_2$};
		\draw[->] (S) -- (I12) ;
		\draw[->,] (I2) -- (I12);
		\draw[->,] (I1) -- (I12);
		\draw[->,] (I12) -- (R);
		\draw [->] (S)-- ++(1.3,.1) -- ++(0,.0) -|  (I2);
		\draw [->] (S)-- ++(-1,.01) -- ++(0,.1) -|  (I1);
		\draw [->] (I2) -- ++(0,-2) -|  (R.east);
		\draw [->] (I1)-- ++(-.09,-.5) -- ++(0,-1.5) -|  (R.west);
		\node[left =of I1] at (1.5,1.5) {$\alpha_1,\beta_1$};
		\node[left =of I1] at (8.5,1.5) {$\alpha_2,\beta_2$};
		\node[left =of I12]at (3.5,0.25) {$\eta_1,\gamma_1$};
		\node[left =of I1] at (5.3,1) {$\alpha_3$};
		\node[left =of I1] at (4.5,-0.9) {$\rho_3$};
		\node[left =of I1] at (6.5,0.25) {$\eta_2,\gamma_2$};
		\node[left =of I2] at (3,-2.5) {$\rho_1$};
		\node[left =of I1] at (7,-2.5) {$\rho_2$};
		
		\end{tikzpicture}
	\end{center}
	\vspace*{-0.5cm}
	\caption{ Flow diagram for two strains coinfection model.}
	\label{flowdiag}
\end{figure}
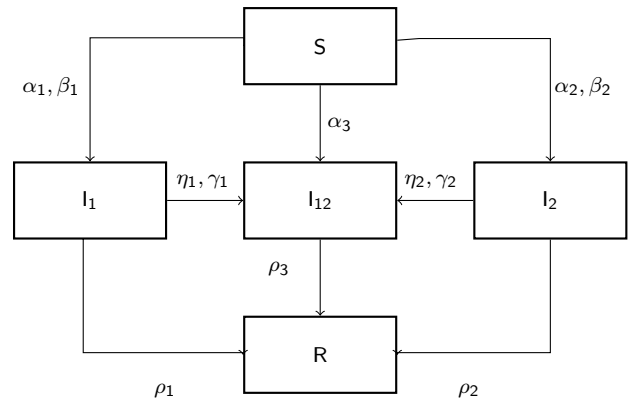

\noindent

Following \cite{SKTW18a,Allen,Bremermann,Zhou}, we assume limited population growth by making the per capita reproduction rate depend on the density of population. We also consider the recovery of each infected class. The SIR model is then described by the system of five ordinary differential equations as follows
\begin{equation}\label{Model}
\begin{split}
S' &=\biggl(b_1(1-\frac{N}{K_1})-\alpha_1I_1-\alpha_2I_2-(\beta_1+\beta_2+\alpha_3)I_{12}-\mu_0\biggr)S,\\
I_1' &=\left(b_2(1-\frac{N}{K_2})+\alpha_1S - \eta_1I_{12}-\gamma_1I_2 - \mu_1\right)I_1+\beta_1SI_{12},\\
I_2' &=\left(b_3(1-\frac{N}{K_3})+\alpha_2S - \eta_2I_{12}-\gamma_2I_1- \mu_2\right)I_2+\beta_2SI_{12},\\
I_{12}' &=\left(b_4(1-\frac{N}{K_4})+\alpha_3S+ \eta_1I_1+\eta_2I_2-\mu_3\right)I_{12}+(\gamma_1+\gamma_2)I_1I_2, \\
R' &=\left(b_5(1-\frac{N}{K_5})-\mu_4'\right)R+\rho_1 I_1+\rho_2I_2+\rho_3 I_{12}.
\end{split}
\end{equation}
Here $S$ represents the susceptible class, $I_1$ and $I_2$ are infected classes from strain 1 and strain 2 respectively, $I_{12} $ represents co-infected class, $R$ represents the recovered class. Finally,
$$N=S+I_1+I_2+I_{12}+R
$$
is the total population. Here
\begin{itemize}
\item
$b_i$ is the birthrate of class $i=1,2,3,4$;
\item
$K_i$ is carrying capacity of class $i=1,2,3,4$;
\item
$\rho_i$ is recovery rate from each infected class ($i=1,2,3$);
\item
$\beta_i$ is the rate of transmission of single infection from coinfected class ($i=1,2$);
\item
$\gamma_i$ is the rate at which infected with one strain get infected with  the other strain and move to coinfected class ($i=1,2$);
\item
$\mu_i' $ is death rate of each class,  $( i=1,2,3,4)$;
\item
$\alpha_1$, $\alpha_2$, $\alpha_3$  are rates of transmission of strain 1, strain 2 and both strains (in the case of coinfection),
\item
$\eta_i$ is rate at which infected from one strain getting  infection from co-infected class $( i=1,2)$ and $\mu_i=\rho_i+\mu_i', i=1,2,3.$
\end{itemize}

In the present paper we address a certain specialization of \eqref{Model} (a SIR model with limited growth of susceptible population). More precisely, we assume that the per capita reproduction of susceptible population is limited by carrying capacity $K<\infty$ while infected and recovered individuals cannot reproduce, i.e.
$$
b_i=0\quad \text{for }i\ge2.
$$
The corresponding SIR model reduces to the following system:
\begin{equation}\label{submodel1}
\begin{split}
S' &=(b(1-\frac{S}{K})-\alpha_1I_1-\alpha_2I_2-(\beta_1+\beta_2+\alpha_3)I_{12}-\mu_0)S,\\
I_1' &=(\alpha_1S - \eta_1I_{12}-\gamma_1I_2 - \mu_1)I_1+\beta_1SI_{12},\\
I_2' &=(\alpha_2S - \eta_2I_{12}-\gamma_2I_1- \mu_2)I_2+\beta_2SI_{12},\\
I_{12}' &=(\alpha_3S+ \eta_1I_1+\eta_2I_2-\mu_3)I_{12}+(\gamma_1+\gamma_2)I_1I_2, \\
R' &=\rho_1 I_1+\rho_2I_2+\rho_3 I_{12}-\mu_4' R.
\end{split}
\end{equation}

We consider some  natural assumptions on the fundamental parameters  of the system and the initial data. First note that if the reproduction rate of susceptible is less than their death rate then  population will die out quickly. So we shall always  assume that
$$
b>\mu_0.
$$
The system is considered under the natural initial conditions
\begin{equation}\label{initdata}
S(0)>0,\quad I_1(0)\geq0, \quad I_2(0)\geq0,\quad  I_{12}(0)\geq0.
\end{equation}
Then it  follows from the standard theory (see, for example, Proposition ~2.1 in \cite{Haddad})  that  any integral curve with \eqref{initdata} is staying in the non negative cone for all $t\ge0$. Note also that since the variable $R$ is not present in the first four equations, we may consider only the first four equations of system \eqref{Model}.
In \cite{SKTW18a} the particular case of \eqref{submodel1} was completely treated, when the parameters $\beta_i$ and $\gamma_j$ vanish, i.e.
\begin{equation}\label{vanish}
\beta_1=\beta_2=\gamma_1=\gamma_2=0,
\end{equation}
The corresponding system has the Lotka-Volterra type and an approach based on the linear complimentary problem was suggested. The latter allows to obtain an effective description of the transition dynamics of  \eqref{submodel1} for any admissible values of the fundamental parameters. The present model is more involved and is no longer a Lotka-Volterra system. But, thinking of \eqref{submodel1} as a perturbation  of the Lotka-Volterra case it is reasonable to believe that some basic properties can be extended for positive small values $\beta_i$ and $\gamma_j$. Below we shall see that this is indeed the case.

Let us introduce the reproduction/threshold numbers of the system \eqref{submodel1} for strains $1$ and $2$ respectively by
$$
\sigma_i:=\frac{\mu_i}{\alpha_i}, \qquad 1\le i\le 2.
$$
Then by change of the indices (if needed) we may assume that $\sigma_1 < \sigma_2$, i.e.
$$
\sigma_1 < \sigma_2.
$$
The letter also means that strain 1 is more aggressive than strain 2.

Since the interaction between coinfected and susceptible classes results a single infection transmission or simultaneous transmission of two infections:
\begin{align*}
\beta_1SI_{12}& \quad \rightarrow \quad I_1\\
\beta_2SI_{12}& \quad \rightarrow \quad I_2\\
\alpha_3SI_{12}& \quad \rightarrow \quad I_{12}
	\end{align*}
it follows from \figurename{1} that the corresponding reproduction/threshold number of the coinfected class is determined by
\begin{equation}\label{sigma3}
\sigma_3=\frac{\mu_3}{\hat{\alpha_3}}=\frac{\mu_3}{\alpha_3+\beta_1+\beta_2}.
\end{equation}
where
\begin{equation}\label{hatalpha}
\hat{\alpha_3}:=\alpha_3+\beta_1+\beta_2
\end{equation}
is  the total transmission rate of infection from coinfected class to susceptible class. Note that in the Lotka-Volterra case \eqref{vanish}
 $\sigma_3=\frac{\mu_3}{\alpha_3}$ which is completely consistent with the notation of \cite{SKTW18a}.

For biological reasons, the latter total transmission rate should be less than other transmission rates comparable with the corresponding death rates. On the other hand, it is natural to assume that the death rates $\mu_i$ are almost the same for different classes. This  makes it natural to assume the following hypotheses:
\begin{equation}\label{assum}
\sigma_1< \sigma_2 < \sigma_3.
\end{equation}

Finally, let us introduce an important parameter of the above system, the so-called \textit{modified carrying capacity} defined by
\begin{equation}\label{capac}
S^{**}:=K(1-\frac{\mu_0}{b})>0.
\end{equation}
Note that $S^{**}$ is always less than $K$ but it is proportional to $K$ whenever $b$ and $\mu_0$ are fixed. We study the transition dynamics of stable equilibrium states depending on the value of $K$ in section~\ref{sec:transition} below.
The vector of fundamental parameters
$$
p=(b,K,\mu_i,\alpha_j,\eta_k,\gamma_k,\beta_k)\in R^{11}_+,\qquad
0\le i\le 3, \,\,1\le j\le 3, \,\,1\le k\le 2,
$$ is said to be \textit{admissible} if \eqref{assum} holds.

\section{General facts on the SIR model }
\eqref{submodel1}
\label{sec:general}
In this section,  we study some basic properties for the system \eqref{submodel1} which are essential in the our analysis of stability results. We follow  the approach given in \cite{SKTW18a} for the Lotka-Volterra case \eqref{vanish}.
\begin{proposition}
	If $(S,I_1,I_2,I_{12})(t)$ is a solution of \eqref{submodel1} with $S(0)$ positive then
	\begin{equation}\label{colo}
	S(t) \leq \frac{1}{\frac{1}{S^{**}}(1-e^{-(b-\mu_0)t})+\frac{1}{S(0)}e^{-(b-\mu_0)t}} .
	\end{equation}
	In particular,
	\begin{equation}\label{boundr1}
	S(t)\leq \max\{ S^{**},S(0)\}
	\end{equation}
	and
	\begin{equation}\label{limit}
	\limsup_{t\rightarrow \infty} S(t)\leq  S^{**}.
	\end{equation}
\end{proposition}
\begin{proof}
	It follows from the first equation of \eqref{submodel1}
	\begin{equation*}
	S'-(b-\mu_0)S \leq-\frac{bS^2}{K},
	\end{equation*}
	which can be written as
	\begin{equation*}
	(Se^{-(b-\mu_0)t})' \leq - \frac{b}{K} e^{-(b-\mu_0)t}S^2.
	\end{equation*}
	
	Dividing both sides by $(S'e^{-(b-\mu_0)t})^2$ and integrating from $0$ to $t$ gives,
	\begin{equation*}
	\frac{e^{(b-\mu_0)t}}{S} \geq \frac{b}{K(b-\mu_0)} (e^{(b-\mu_0)t}-1) + \frac{1}{S(0)},
	\end{equation*}
	which yields \eqref{colo}. Then relations \eqref{boundr1} and \eqref{limit} follow immendiately from \eqref{colo}.
\end{proof}
Another important property of the general system \eqref{submodel1} is the following.

\begin{proposition}[Global estimates]\label{pro:glob}
	If $(S,I_1,I_2,I_{12})(t)$ is a solution of \eqref{submodel1} with positive initial data  
	then
	\begin{equation}
	S(t)+I_1(t)+I_2(t)+I_{12}(t) \leq \max\{S(0)+I_1(0)+I_2(0)+I_{12}(0),\frac{bS_m}{\mu}\}
	\end{equation}
	for $t \geq 0$, where $\mu=\min_{0\le i\le 3} \mu_i$ and $S_m= \max\{ S^{**},S(0)\}.$
\end{proposition}
\begin{proof}
Summing up the first four equations of \eqref{submodel1} we obtain
	\begin{equation*}
	S'+I_1'+I_2'+I_{12}' \leq (b-\frac{bS}{K})S-\mu(S+I_1+I_2+I_{12}).
	\end{equation*}
	Let  $y(t)=S+I_1+I_2+I_{12}$, then
	\begin{equation*}
	y' \leq bS_m-\mu y.\\
	\end{equation*}
	Multiplying both sides by $e^{\mu t}$ and integrating the above equation from $0$ to $t$ gives
	\begin{equation*}
		y(t)  \leq e^{-\mu t}y(0)+\frac{bS_m}{\mu}(1-e^{-\mu t}).
	\end{equation*}
	By \eqref{boundr1}, we have $	y(t)  \leq \max \{y(0),\frac{bS_m}{\mu}\}$ which proves our claim.
\end{proof}

Finally, in the global stability analysis given below in section~\ref{sec:global}, we shall need the following result established recently in \cite{SKTW18a}.
\begin{proposition}\label{G.Sprop}
Suppose that $f(t) \in L^p ([0,\infty))\cap C^1([0,\infty))$, where $p\geq 1,$ and the first $k$ higher derivatives are bounded: $f',  \ldots, f^{(k)}\in L^\infty([0,\infty))$. Then
	  $$
	  \lim_{t\to\infty}f(t)=\ldots= \lim_{t\to\infty}f^{(k-1)}(t)=0.$$
\end{proposition}

\section{Equilibrium points: the local stability analysis}\label{eqpoint}

\subsection{Basic properties}
In this section we identify all equilibria of the system \eqref{submodel1} in the case when
\begin{equation}\label{nonvanish}
\beta_1>0, \quad \beta_2>0,\quad \gamma_1+\gamma_2>0
\end{equation}
and determine their local stability properties. First, let us remark some useful balance relations which hold for any equilibrium point $Y=(Y_0,Y_1,Y_2,Y_3)$ of \eqref{submodel1}. Since we are only interested in nonnegative equilibrium states we always assume that
$$Y=(Y_0,Y_1,Y_2,Y_3)\ge0.
$$
Then we have
\begin{equation}\label{equil}
\begin{split}
(b(1-\frac{Y_0}{K})-\alpha_1Y_1-\alpha_2Y_2-(\beta_1+\beta_2+\alpha_3)Y_3-\mu_0)Y_0&=0\\
(\alpha_1Y_0 - \eta_1Y_3-\gamma_1Y_2 - \mu_1)Y_1+\beta_1Y_0Y_3&=0\\
(\alpha_2Y_0 - \eta_2Y_3-\gamma_2Y_1- \mu_2)Y_2+\beta_2Y_0Y_3&=0\\
(\alpha_3Y_0+ \eta_1Y_1+\eta_2Y_2-\mu_3)Y_3+(\gamma_1+\gamma_2)Y_1Y_2&=0 \end{split}
\end{equation}
Denote by $G_1:=(0,0,0,0)$ the trivial equilibrium state. Then
\begin{equation}\label{claim}
Y_0\ne0 \quad \text{unless $Y=G_1$.}
\end{equation}
 Indeed, if $Y_0=0$ then we have from the second equation of \eqref{equil} that $(\eta_1Y_3+\gamma_1Y_2 + \mu_1)Y_1=0$. But by the positivity assumption, $\eta_1Y_3+\gamma_1Y_2 + \mu_1\ge \mu_1>0$, hence $Y_1=0$. For the same reason we have $Y_2=0$, thus the last equation in \eqref{equil} yields
 $$\mu_3Y_3=(\gamma_1+\gamma_2)Y_1Y_2=0,
 $$
 hence $Y_3=0$ too. This proves that $Y=G_1$ and proves \eqref{claim}.

It follows that any nontrivial equilibrium state $Y\ne G_1$ must satisfy
$$
b(1-\frac{Y_0}{K})-\alpha_1Y_1-\alpha_2Y_2-(\beta_1+\beta_2+\alpha_3)Y_3-\mu_0=0
$$
which implies by \eqref{capac} and \eqref{hatalpha} the balance equation
\begin{equation}
	\alpha_1Y_1+\alpha_2Y_2+\hat{\alpha_3}Y_{3}=\frac{b}{K}(S^{**}-Y_0)
	\label{law1}
\end{equation}
Also, summing up  equations in \eqref{equil} we obtain
	\begin{eqnarray}
	\mu_1Y_1+\mu_2Y_2+\mu_3Y_{3}&=&\frac{b}{K}(S^{**}-Y_0)Y_0.
	\label{law2}
	\end{eqnarray}

Taking into account \eqref{claim}, the latter identities imply a priori bounds for $Y_0$:
\begin{lemma}\label{lem:equi}
Let $Y\ne G_1$ be a nontrivial equilibrium point of \eqref{submodel1}. Then
	\begin{equation}\label{Ystar}
	0<Y_0\le S^{**},
	\end{equation}
	and the  equality holds if and only if $Y_1=Y_2=Y_3=0$. Furthermore,
	\begin{equation}\label{Kineq}
	\sigma_1\le Y_0\le \min\{S^{**},\sigma_3\},
	\end{equation}
	unless $Y_0=S^{**}$.

\end{lemma}

\begin{proof}
Indeed, the first claim follows immediately from \eqref{law1}. Next, assuming that $Y_0\ne S^{**}$ and dividing  \eqref{law2} by  \eqref{law1} we get $Y_0=\frac{\mu_1Y_1+\mu_2Y_2+\mu_3Y_{3}}{\alpha_1Y_1+\alpha_2Y_2+\hat{\alpha_3}Y_{3}}$ which readily yields \eqref{Kineq}.
\end{proof}

Combining the above estimates we obtain from \eqref{law1} the following a priori bound on the coordinates of an arbitrary equilibrium point $Y$ distinct from $G_2:=(S^{**},0,0,0)$  that
	\begin{equation}\label{max}
	\|Y\|_{\infty}:=\max_{0\le i\le 3}Y_i\le \max\left\{\frac{b-\mu_0}{\alpha_1},\frac{b-\mu_0}{\alpha_2}, \frac{b-\mu_0}{\hat\alpha_3},\sigma_3\right\}
	\end{equation}

\subsection{Equilibrium points of \eqref{submodel1}}
Note that \eqref{submodel1} always has the trivial equilibrium state
$$
G_1=(0,0,0,0).
$$
The first nontrivial equilibrium point is the disease-free equilibrium (or a healthy equilibrium) which is an equilibrium such that the disease is absent in all the patches. In the present notation the disease-free equilibrium corresponds to $I_1=I_2=I_{12}=0$, and it follows from \eqref{submodel1} that
$$
G_2=(S^{**},0,0,0).
$$
Then Lemma~\ref{lem:equi} shows that the value of the susceptible class for the healthy equilibrium state $G_2$ is \textit{the largest possible among all equilibrium points.}

Suppose that $Y_3=0$. Then it follows from \eqref{equil} that besides $G_2$ there exist exactly two equilibrium points with the presence of the first  or the second strains, given respectively by
\begin{equation}\label{eqpt3}
G_3  =\left(\sigma_1,\frac{b}{K\alpha_1}(S^{**}-\sigma_1), 0,0\right)\quad\text{when}\quad S^{**}>\sigma_1.
\end{equation}
and
\begin{equation}\label{eqpt4}
G_4 =\left(\sigma_2,0,\frac{b}{K\alpha_2}(S^{**}-\sigma_2), 0\right)\quad\text{when}\quad S^{**}>\sigma_2.
\end{equation}

Finally suppose that $Y$ be a nontrivial equilibrium point such that $Y_3\ne0$. Since $Y_0\ne0$, the second and the third equations in \eqref{equil} immediately imply that $Y_1\ne0$ and $Y_2\ne0$ as well. Thus $Y_3\ne0$ implies that $Y$ must have all positive coordinates. This equilibrium point is related to the coexistence of both strains with coinfection and is called the coexistence equilibrium point. We shall denote it by
$$
G_5=(S^*,I_1^*,I_2^*,I_{12}^*).
$$
Note that due to the complexity of our model it is difficult to find the coordinates of the $G_5$-type equilibrium points explicitly. It is also a priori unclear how many such coexistence equilibrium points can exist. We address this issue in a forthcoming paper.


\subsection{The trivial equilibrium point $G_1$}
The Jacobian matrix for $G_i=(0,0,0,0)$ is
$$
J=
\begin{bmatrix}
b-\mu_0 & 0 &0 & 0  \\
0 & -\mu_1 & 0 & \\
0 & 0 & -\mu_2 & 0 \\
0 & 0 & 0 & -\mu_3
\end{bmatrix}
$$
Since $b-\mu_0>0$, the trivial equilibrium point $G_1$ is always locally unstable. In fact, we have a  stronger observation.
\begin{proposition}
	If
	\begin{equation}\label{insE1}
	\kappa:=\limsup_{t \rightarrow \infty}(\alpha_1I_1+\alpha_2I_2+\hat\alpha_3 I_{12})<(b-\mu_0),
	\end{equation}
	then
	\begin{equation}\label{insE11}
	\liminf_{t \rightarrow \infty} S(t) \geq \frac{K}{b}(b-\mu_0-\kappa).
	\end{equation}
In particular, if $I_1,I_2,I_{12}\to 0$ as $t\to \infty$ then $S(t)$ is separated from zero.
\end{proposition}

\begin{proof}
By virtue of \eqref{insE1} we have for every $\kappa_1 \in (\kappa,(b-\mu_0))$ that there exist $t_1>0$ such that
$$\alpha_1I_1+\alpha_2I_2+\hat\alpha_3 I_{12} \leq \kappa_1
$$
holds for any $t \geq t_1$.
	It follows from the first equation of \eqref{submodel1} that
	$$S'(t)-(b-\mu_0-\kappa_1)S(t) \geq-\frac{bS(t)^2}{K},$$ for $t\geq t_1$,
	hence
	\begin{equation*}
	(S(t)e^{-(b-\mu_0-\kappa_1)t})' \geq - \frac{b}{K} e^{-(b-\mu_0-\kappa_1)t}S(t)^2.
	\end{equation*}
	Dividing both sides by $(S'e^{-(b-\mu_0-\kappa_1)t})^2$ and integrate from $t_1$ to $t$ gives,
		\begin{equation*}
	\frac{e^{(b-\mu_0-\kappa_1)(t-t_1)}}{S(t)} \leq \frac{b}{K(b-\mu_0-\kappa_1)} (e^{(b-\mu_0-\kappa_1)(t-t_1)}-1) + \frac{1}{S(t_1)},
		\end{equation*}
	which leads to
	\begin{equation*}
	S(t) \geq \frac{S(t_1)}{\frac{b}{K(b-\mu_0-\kappa_1)}S(t_1)(1-e^{-(b-\mu_0-\kappa_1)(t-t_1})+e^{-(b-\mu_0-\kappa_1)(t-t_1)}},\quad\text{for}\quad t \geq t_1
		\end{equation*}
	and gives
	$$\liminf_{t \rightarrow \infty} S(t) \geq \frac{K}{b}(b-\mu_0-\kappa_1).$$
	Since $ \kappa_1$ is an arbitrary number from $(\kappa,b-\mu_0),$ this implies \eqref{insE11}.
\end{proof}

\subsection{The disease-free equilibrium $G_2$}

As we know by Lemma~\ref{lem:equi}, the disease-free equilibrium $G_2=( S^{**}, 0,0,0)$ has the largest possible among all equilibrium points. The Jacobian matrix for $G_2=(S^{**},0,0,0)$ is
\[
J=
\begin{bmatrix}
-(b-\mu_0) & -\alpha_1S^{**} & -\alpha_2S^{**} & -\hat{\alpha_3}S^{**}  \\
0 & \alpha_1S^{**}-\mu_1 & 0 & \beta_1S^{**} \\
0 & 0 & \alpha_2S^{**}-\mu_2 & \beta_2S^{**} \\
0 & 0 & 0 & \alpha_3S^{**}-\mu_3
\end{bmatrix},
\]
and has all negative eigenvalues if $S^{**}\leq\sigma_1$ and \eqref{assum} holds. This implies

\begin{proposition}\label{pro:G2}
The disease free equilibrium point $G_2$  is locally  stable whenever $S^{**}\leq\sigma_1$ holds.
\end{proposition}

\subsection{The equilibrium point with the presence of the first strain $G_3$}\label{sec:G3loc}
The local stability analysis  of equilibrium points $G_3$ and $G_4$ is more involved. First note that it follows from \eqref{eqpt3} that $G_3$ is nonnegative if and only if
$$
I_1^*:=\frac{b}{K\alpha_1}(S^{**}-\sigma_1)\ge0,
$$
and, moreover, $G_3=G_2$ when $I_1^*=0$.  Using \eqref{eqpt3}, we find the corresponding  Jacobian matrix evaluated at $G_3$:
$$
J=\begin{bmatrix}
A&\star\\
0&B
\end{bmatrix}=
\begin{bmatrix}
-b\frac{\sigma_1}{K} & -\alpha_1\sigma_1 & -\alpha_2\sigma_1 & -\hat{\alpha_3} \sigma_1  \\
\alpha_1I^*_1 & 0 & -\gamma_1I^*_1 & -\eta_1I^*_1+\beta_1\sigma_1 \\
0 & 0 & -\alpha_2(\sigma_2-\sigma_1)-\gamma_2I^*_1 & \beta_2\sigma_1 \\
0 & 0 & (\gamma_1+\gamma_2)I^*_1 & \alpha_3\sigma_1+\eta_1I^*_1-\mu_3
\end{bmatrix}
$$
where we partitioned  the Jacobian matrix into $2\times2$ blocks with
\begin{align*}
A&=
\begin{bmatrix}
-b\frac{\sigma_1}{K} & -\alpha_1\sigma_1\\
\alpha_1I^*_1 & 0
\end{bmatrix},
\qquad
B=
\begin{bmatrix}
-\alpha_2(\sigma_2-\sigma_1)-\gamma_2I^*_1 & \beta_2\sigma_1\\
(\gamma_1+\gamma_2)I^*_1 & \alpha_3\sigma_1+\eta_1I^*_1-\mu_3
\end{bmatrix}
\end{align*}
It follows that the equilibrium point $G_3$ is stable if and only if both $A$  and $B$ are stable. We have for the first block matrix
$$
\trace A=-b\frac{\sigma_1}{K} <0, \qquad \det A=\alpha_1^2 I_1^*\sigma_1>0,
$$
therefore $A$ is stable for any choice of parameters provided that $G_3$ exists and distinct of $G_2$ (i.e. $I_1^*>0$).

Next notice that the matrix $B$ is stable if and only if its trace is negative and the determinant is positive. Since the first diagonal element in $C$ is negative by  \eqref{assum}, and the anti-diagonal elements are positive, the positivity of the determinant implies that
\begin{equation}\label{Ccond}
B_{22}=\alpha_3\sigma_1+\eta_1I^*_1-\mu_3<0.
\end{equation}
Therefore, $\det B>0$  implies that $\trace B=B_{11}+B_{22}<0$, i.e. $B$ is stable. This shows that the block $B$ is stable if and only if  $\det B>0$ holds. In summary, we have

\begin{proposition}\label{pro:G3}
The equilibrium point with the presence of the first strain $G_3$ is locally  stable if and only if $I_1^*>0$ and
\begin{equation}\label{detB}
\det B=\det\begin{bmatrix}
-\alpha_2(\sigma_2-\sigma_1)-\gamma_2I^*_1 & \beta_2\sigma_1\\
(\gamma_1+\gamma_2)I^*_1 & \alpha_3\sigma_1+\eta_1I^*_1-\mu_3
\end{bmatrix}>0.
\end{equation}
\end{proposition}

\medskip

Let us consider the determinantal condition \eqref{detB} in more detail. In the Lotka-Volterra case \eqref{vanish} treated in \cite{SKTW18a}, one has $\gamma_i=\beta_j=0$, hence the corresponding determinant condition
$$
\det B_0=\det\begin{bmatrix}
-\alpha_2(\sigma_2-\sigma_1)& 0\\
0 & \alpha_3\sigma_1+\eta_1I^*_1-\mu_3
\end{bmatrix}>0
$$
becomes equivalent to a simpler inequality (cf. with \eqref{Ccond})
$$
I^*_1=\frac{b}{K\alpha_1}(S^{**}-\sigma_1)<\frac{\mu_3-\alpha_3\sigma_1}{\eta_1}.
$$
Coming back to the general case \eqref{nonvanish}, the determinant
$$
\Delta(\lambda):=\det\begin{bmatrix}
-\alpha_2(\sigma_2-\sigma_1)-\gamma_2\lambda & \beta_2\sigma_1\\
(\gamma_1+\gamma_2)\lambda & \alpha_3\sigma_1+\eta_1\lambda-\mu_3
\end{bmatrix}
$$
is a quadratic polynomial in $\lambda$ with a negative leading coefficient. Further, by virtue of \eqref{sigma3} and \eqref{assum} we have
$$
\Delta(0)=\alpha_2\alpha_3(\sigma_2-\sigma_1) \left(\frac{\mu_3}{\alpha_3}-\sigma_1\right)>
\alpha_2\alpha_3(\sigma_2-\sigma_1) \left(\sigma_3-\sigma_1\right)>0,
$$
hence the equation $\Delta(\lambda)$ has a unique positive root. We denote it by $\Lambda$. Then one can easily see that \eqref{detB} is equivalent to the inequality
\begin{equation}\label{I1ineq}
0<I_1^*=\frac{b}{K\alpha_1}(S^{**}-\sigma_1)<\Lambda.
\end{equation}
which is equivalent to
\begin{equation}
\sigma_1<S^{**}<\sigma_1+\frac{K\alpha_1}{b}\Lambda
\end{equation}
and under this condition the equilibrium point $G_3$ is stable.


When $\Delta(I_1^*)\approx 0$ and negative, it is plausible to expect that $G_3$ bifurcates into an inner point of $G_5$-type. We shall consider this question in short in section~\ref{sec:transition} below.

\subsection{The equilibrium point with the presence of the second strain $G_4$}
Similarly to the above, $G_4$ is nonnegative if and only if
$$
I_2^*:=\frac{b}{K\alpha_2}(S^{**}-\sigma_2)\ge0.
$$
Note that if $G_4$ is nonnegative then $S^{**}\ge \sigma_2$, hence by virtue of \eqref{assum}, $G_3$ is nonnegative too. The Jacobian matrix evaluated at $G_4$ is
\begin{equation}\label{JG4}
J=
\begin{bmatrix}
-b\frac{\sigma_2}{K} & -\alpha_1\sigma_2 & -\alpha_2\sigma_2 & -\hat{\alpha_3} \sigma_2  \\
0 & \alpha_1(\sigma_2-\sigma_1)-\gamma_1I_2^* & 0 & \beta_1\sigma_2 \\
\alpha_2I_2^* & -\gamma_2I_2^* & 0 & -\eta_2I_2^*+\beta_2\sigma_2 \\
0 & (\gamma_1+\gamma_2)I_2^* & 0 & -\mu_3+\alpha_3\sigma_2+\eta_2I_2^*
\end{bmatrix}
\end{equation}
Note that,  the elementary row operation  of the matrix \eqref{JG4} do not affect the eigenvalues of this matrix. Therefore, after an obvious rearrangement, the stability of $J$ is equivalent to that of the following matrix
\begin{equation*}
\tilde{J}=\begin{bmatrix}
C&\star\\
0&D
\end{bmatrix}
\end{equation*}
with the diagonal $2\times 2$-blocks
\begin{equation}\label{J2G4}
C=\begin{bmatrix}
-b\frac{S}{K} & -\alpha_2\sigma_2\\
\alpha_2I_2^* & 0
\end{bmatrix},
\qquad
D=\begin{bmatrix}
 \alpha_1(\sigma_2-\sigma_1)-\gamma_1I_2^*  & \beta_1\sigma_2 \\
(\gamma_1+\gamma_2)I_2^* & -\mu_3+\alpha_3\sigma_2+\eta_2I_2^*
\end{bmatrix}
\end{equation}
So $J$ is stable if and only if the blocks \eqref{J2G4} are stable. The first block $C$ is stable provided $G_4$ is nonnegative. Thus, the stability of $G_4$ is equivalent to that of $D$. Similarly to the previous case, we have

\begin{proposition}\label{pro:G4}
The equilibrium point with the presence of the second strain $G_4$ is locally stable if and only if $S^{**}>\sigma_2$ and
\begin{equation}\label{detD}
\det D=\det\begin{bmatrix}
 \alpha_1(\sigma_2-\sigma_1)-\gamma_1I_2^*  & \beta_1\sigma_2 \\
(\gamma_1+\gamma_2)I_2^* & -\mu_3+\alpha_3\sigma_2+\eta_2I_2^*
\end{bmatrix}>0
\end{equation}
and
\begin{equation}\label{traceD}
\trace D=
 \alpha_1(\sigma_2-\sigma_1)-\gamma_1I_2^*   -\mu_3+\alpha_3\sigma_2+\eta_2I_2^*
<0.
\end{equation}

\begin{remark}
Note that if \eqref{vanish} is satisfied then $D$ is upper triangular and has an eigenvalue $\alpha_1(\sigma_2-\sigma_1)>0$, thus unstable. This shows that for the local stability of $G_4$, $\gamma_1$ must be larger an a priori lower bound. Since we consider \eqref{submodel1} as a suitable modification of the Lotka-Volterra case \eqref{vanish}, $G_4$ is an unstable point for small perturbations of the parameters $\gamma_i$ and $\beta_j$. The latter is completely consistent with the results of \cite{SKTW18a}.
\end{remark}

\end{proposition}

%

\section{Global Stability of Equilibrium points}\label{sec:global}

\subsection{The Lyapunov function}
Let us denote
$$
Y=  \left(
 \begin{array}{c}
   Y_0 \\
   Y_1\\
   Y_2 \\
   Y_3
 \end{array}
 \right)
 =
 \left(
 \begin{array}{c}S\\
 I_1\\I_2
 \\I_{12}
 \end{array}
 \right)
$$
and rewrite our system \eqref{submodel1} in this notation
\begin{equation}\label{gen}
 \begin{split}
 \frac{dY_k}{dt}&=F_{k}(Y)\cdot Y_k+H_k,\qquad k=0,1,2,3,
 \end{split}
 \end{equation}
  where we denote
   \begin{equation}\label{gen10}
   F(Y)=-q+AY,
   \end{equation}
  with
 \begin{equation}\label{gen11}
 F(Y)=
 \left(
 \begin{array}{c}
   F_0(Y)\\
  F_1(Y)\\
   F_2(Y)\\
   F_3(Y)
 \end{array}
 \right)
,\quad
q=
 \left(
 \begin{array}{c}
   -b+\mu_0 \\
   \mu_1\\
   \mu_2 \\
   \mu_3
 \end{array}
 \right)
,\quad
A=
  \left(
 \begin{array}{cccc}
   -\frac{b}{K}&-\alpha_1&-\alpha_2&-\hat{\alpha}_3 \\
   \alpha_1&0&-\gamma_2&-\eta_1\\
   \alpha_2&-\gamma_1&0&-\eta_2 \\
   \alpha_3&\eta_1&\eta_2&0
 \end{array}
 \right)
 ,\quad
 H=
 \left(
 \begin{array}{c}
   0\\
   \beta_1Y_0Y_3\\
   \beta_2Y_0Y_3\\
   (\gamma_1+\gamma_2)Y_1Y_2
 \end{array}
 \right)
 \end{equation}
Then $Y^*=(Y^*_0,Y^*_1,Y^*_2,Y^*_3)$ is an equilibrium point of \eqref{submodel1} if and only if
 \begin{equation}\label{FY0}
Y^*_iF_i(Y^*)=-H_i(Y^*), \qquad 0\le i\le 3.
\end{equation}

We associate with $Y^*$ the Lyapunov function
$$
v_{Y^*}(Y_0,Y_1,Y_2,Y_3)=\sum_{i=0}^3 (Y_i-Y_i^*\ln Y_i).
$$
The derivative computations are slightly different from the Lotka-Volterra case considered previously and the resulting function contains both the $\gamma_i$, $\beta_j$ and the $H$-terms. More precisely, using consequently \eqref{gen10}, \eqref{gen11} and \eqref{FY0}  we obtain for the time derivative of $v_{Y^*}$ along any integral trajectory of \eqref{gen}
\begin{equation}
\begin{split}\label{Fequa}
\frac{d}{dt}v_{Y^*}:= (\nabla v_{Y^*})^T\,\frac{dy}{dt}
=& \sum_{i=0}^3 \frac{Y_i-Y_i^*}{Y_i}(F_i(Y)Y_i+H_i(Y))\\
=& \sum_{i=0}^3 \frac{Y_i-Y_i^*}{Y_i}((F_i(Y)-F_i(Y^*))Y_i+F_i(Y^*)Y_i+H_i(Y))\\
=& \sum_{i,j=0}^3 A_{ij}(Y_i-Y_i^*)(Y_j-Y^*_j)+
\sum_{i=0}^3(Y_i-Y_i^*)F_i(Y^*)+\sum_{i=1}^3\frac{Y_i-Y_i^*}{Y_i}H_i(Y)\\
=&
-\frac{b}{K}(Y_0-Y_0^*)^2-(\gamma_1+\gamma_2)(Y_1-Y_1^*)(Y_2-Y_2^*)-
(\beta_1+\beta_2)(Y_0-Y_0^*)(Y_3-Y_3^*)\\
&+\sum_{i=0}^3Y_iF_i(Y^*)+\sum_{i=1}^3H_i(Y^*)+\frac{Y_i-Y_i^*}{Y_i}H_i(Y)\\
\end{split}
\end{equation}
Using \eqref{gen11}, we find further that
\begin{equation}\label{Lyap1}
\frac{d}{dt}v_{Y^*}=
-\frac{b}{K}(Y_0-Y_0^*)^2+\sum_{i=0}^3Y_iF_i(Y^*)+  \Phi,
\end{equation}
where
\begin{equation}\label{Lyap2}
\Phi=(\gamma_1+\gamma_2)\left(Y_1Y_2^*+Y_1^*Y_2-\frac{Y_3^*Y_1Y_2}{Y_3}- Y_1^*Y_2^*\right)
-Y_0Y_3\left(\beta_1\frac{Y_1^*}{Y_1}+\beta_2\frac{Y_2^*}{Y_2}\right)
+(\beta_1+\beta_2)(Y_3Y_0^*+Y_0Y_3^*-Y_0^*Y_3^*)
\end{equation}
Note that in our notation, see \eqref{gen11}, all function $H_i(Y)$ are nonnegative for any nonnegative $Y$, hence it follows from \eqref{FY0} that for any equilibrium point
\begin{equation}\label{FY0neg}
F_i(Y^*)\le0.
\end{equation}
Therefore the sign of the derivative of the Lyapunov function depends on the sign of $\Phi$. Below we apply the above computations to global stability results for the first two equilibrium points.

\subsection{Global stability of equilibrium point $G_2$}
First we consider the disease free equilibrium point. Recall that by Proposition~\ref{pro:G2} the point $G_2$ is \textit{locally} stable if and only if
\begin{equation}\label{conditionG2}
	S^{**}\leq \sigma_1
\end{equation}
holds. Remarkably, the latter condition also implies the global stability. Furthermore, in the $G_2$-case we are able to establish a global stability result which guarantees that the disease can not invade and go extinct in small populations. In particular, the following result shows that disease cannot persist in a small population.

\begin{proposition}\label{G.SG2}
	Let \eqref{conditionG2} be satisfied. Then the equilibrium point $G_2$ is globally asymptotically stable i.e
	\begin{align*}
	\lim_{t\to\infty}I_1(t)&=
	\lim_{t\to\infty}I_2(t)=
	\lim_{t\to\infty}I_{12}(t)=0,\\
	\lim_{t\to\infty} S(t)&= S^{**},
	\end{align*}
\end{proposition}
\begin{proof}
We have for the Lyapunov function of $Y^*=G_2=(S^{**}, 0, 0, 0)$:
	\begin{equation}\label{veq1.1}
	v(t):=S- S^{**}\ln S+I_1+I_2+I_{12}.
	\end{equation}
We have $F_0(G_2)=0$ and $F_i(G_2)=\alpha_i S^{**}-\mu_i$, $1\le i\le 3$, and also $H_i(G_2)=0$ for $1\le i\le 3$. Substituting this into \eqref{Lyap1} and \eqref{Lyap2} we obtain
	\begin{equation}\label{veq1.2}
\begin{split}
v'(t)&=-\frac{b}{K}(S^{**}-S)^2-(\mu_1- \alpha_1S^{**}) I_1-(\mu_2- \alpha_2S^{**}) I_2 -\left(\mu_3-\hat{\alpha_3}S^{**}\right)I_{12}\\
&=-\frac{b}{K}(S^{**}-S)^2-\alpha_1(\sigma_1-S^{**}) I_1-\alpha_2(\sigma_2-S^{**}) I_2 -\hat{\alpha_3}\left(\sigma_3-S^{**}\right)I_{12}\\
&\le0.
\end{split}
\end{equation}
First suppose that we have the strong inequality $S^{**}>\sigma_1$. Then $\sigma_i-S^{**}$ are nonzero and strongly positive for any $i$. Integrating \eqref{veq1.2} over $[0,\infty]$ and using the fact that by Proposition~\ref{pro:glob} all $S$, $Y_1$, $Y_2$, $Y_{12}$ are bounded,   we obtain
%
 $$
 \int_{0}^{\infty}(S-S^*)^2d\tau <\infty,\qquad \text{and}\qquad
 \int_{0}^{\infty}|I_k|d\tau=\int_{0}^{\infty}I_kd\tau <\infty \quad \text{for } I_k=I_1,I_2,I_{12}.
	$$
Using again the boundedness of  $S$, $Y_1$, $Y_2$, $Y_{12}$, in virtue of the system \eqref{submodel1}, their derivative are also bounded. Applying Proposition~\ref{G.Sprop} we conclude that $S$ converges to $S^*$ and $I_1,I_2$, $I_{12}$ converge to zero.	

Now suppose that $S^{**}=\sigma_1$. Since $\sigma_i-S^{**}>0$ for $i=2,3$, the above argument implies that $S$ converges to $S^*$ and $I_2$, $I_{12}$ converge to zero. Let us show that  $\lim_{t\to\infty}I_1=0.$ Since $\lim_{t \rightarrow \infty}S(t)$ exist, then by Proposition~\ref{G.Sprop} , $\lim_{t\to \infty } S'(t) \rightarrow 0$. Therefore the first equation in \eqref{submodel1} implies by virtue of $\lim_{t \rightarrow \infty}S(t)=S^{**}\ne0$ that
	$$
\frac{b}{K}(S(t)-S^{**}) -\alpha_1I_1(t)-\alpha_2I_2(t) -\hat\alpha_3 I_{12}(t) \rightarrow 0,
	$$
	which implies  $\lim_{t\to \infty}I_1(t)=0$ and finishes the proof.
\end{proof}
\begin{remark}
It is interesting to note that when $S^{**}=\sigma_1$  we have no local stability because one eigenvalue is equal to zero. But despite of this due to non-linear character of the problem linear term gives the global asymptotic stability in this case.
\end{remark}

\subsection{Global stability of equilibrium point $G_3$}
\begin{proposition}\label{G.SG3}
	Let
	\begin{equation}\label{g3.1}
	S^{**}>\sigma_1
	\end{equation}
	and
	\begin{equation}\label{conditionG3}
	I_1^*=\frac{b}{K\alpha_1}(S^{**}-\sigma_1)\le \min\left\{
\frac{\alpha_2(\sigma_2-\sigma_1)}{\gamma_1}, \,
\frac{\hat{\alpha}_3(\sigma_3-\sigma_1)}{\eta_1}
\right\},
\end{equation}
then the equilibrium point $G_3$ is globally stable i.e
	\begin{align*}
	\lim_{t\to\infty}I_2(t)&=
	\lim_{t\to\infty}I_{12}(t)=0,\qquad \lim_{t\to\infty} S(t)= S^*,\quad
	\lim_{t\to\infty} I_1(t)=I_1^*
	\end{align*}
\end{proposition}

\begin{remark}\label{R3}
Note that the global stability condition \eqref{conditionG3} coincides with the local stability conditions \eqref{detB} when \eqref{vanish} is satisfied. In the general case, \eqref{conditionG3} implies \eqref{detB}. Indeed, note that in notation of Section~\ref{sec:G3loc}, we have by \eqref{conditionG3}
$$
B_{11}+B_{21}=
-\alpha_2(\sigma_2-\sigma_1)-\gamma_2I^*_1 +(\gamma_1+\gamma_2)I^*_1=
-\alpha_2(\sigma_2-\sigma_1)+\gamma_1I^*_1\le 0,
$$
and similarly
$$
B_{12}+B_{22}=\beta_2\sigma_1+ \alpha_3\sigma_1+\eta_1I^*_1-\mu_3\le
(\beta_2+ \alpha_3)\sigma_1-\hat{\alpha}_3\sigma_1=-\beta_1\sigma_1<0
$$
hence $B_{22}<-B_{12}<0$, and $B_{11}<-B_{21}<0$, which implies
$
B_{22}B_{11}>B_{12}B_{21},
$
therefore $\det B>0$. The latter implies by Proposition~\ref{pro:G3} that $G_3$ is locally stable.
\end{remark}

\begin{proof}
In this case, we $Y^*=G_3=(\sigma_1, I_1^*, 0, 0)$, where $I_1^*:=\frac{b}{K\alpha_1}(S^{**}-\sigma_1)$, and  the corresponding Lyapunov function is given by
	\begin{equation}\label{veq3.1}
	v(t):=S- \sigma_1\ln S+I_1-I_1^*\ln I_1+I_2+I_{12}
	\end{equation}
Substituting this into \eqref{Lyap1} and \eqref{Lyap2} and using \eqref{conditionG3} we obtain
\begin{equation}\label{veq3.2}	
\begin{split}
v'(t)&=-\frac{b}{K}(\sigma_1-S)^2-\left(\alpha_2(\sigma_2-\sigma_1) -\gamma_1I_1^*\right)I_2 -\left(\hat{\alpha}_3(\sigma_3-\sigma_1) -\eta_1I_1^*\right)I_{12} -\beta_1\frac{I_1^*}{I_1}SI_{12}\le 0.
\end{split}
	\end{equation}
We suppose first that the strong inequality in \eqref{conditionG3} holds. Then as above, integrating \eqref{veq3.2} over $[0,\infty]$ we obtain
 $$\int_{0}^{\infty}(S-\sigma_1)^2d\tau <\infty,\qquad
 \int_{0}^{\infty}I_2d\tau <\infty,\quad
 \int_{0}^{\infty}I_{12}d\tau <\infty.
	$$
therefore by Proposition \ref{G.Sprop}, $S$ converges to $\sigma_1$ and $I_2$, $I_{12}$ converge to zero.
Now we consider the convergence of $I_1$.
Arguing as in the proof of Proposition~\ref{G.SG2}, we obtain that $\lim_{t\to \infty}S(t)=\sigma_1$ and since the latter limit is nonzero we find from the first equation in \eqref{submodel1} that
$$
0=\lim_{t\to \infty}\left(\frac{b}{K}(S^{**}-S(t)) -\alpha_1I_1(t)-\alpha_2I_2(t) -\hat\alpha_3 I_{12}(t)\right)=
\frac{b}{K}(S^{**}-\sigma_1) -\alpha_1\lim_{t\to \infty}I_1(t),
$$
this proves that $\lim_{t\to\infty}I_1=I_1^*.$

The case when the equality in \eqref{conditionG3} attains is studied similar to that in the proof of Proposition~\ref{G.SG2}.
\end{proof}

\begin{remark}
Finally remark, that a similar analysis in the case for $G_4$ shows that the corresponding Lyapunov function does not have a negative derivative because  by assumption \eqref{assum} one has
	$$
	\left(\alpha_1(\sigma_2-\sigma_1)+\gamma_2I_2^*\right)>0.
	$$
 This indirectly shows that two infectious agents may not coexist together in the absence of coinfection.
	\end{remark}

\section{Transition dynamics}\label{sec:transition}

As it was already mentioned in Introduction, in \cite{SKTW18a} Lotka-Volterra version of our present model  \eqref{submodel1}, which corresponds to vanishing of the transmission parameters \eqref{vanish}.

In \cite{SKTW18a} a very striking result has been established asserting that for any $K>0$ and \textit{each} admissible choice of the fundamental parameter
$$
\hat{p}=(b,\mu_0,\mu_1,\mu_2,\mu_3,\alpha_1,\alpha_2,\alpha_3,\eta_1,\eta_2)\in R^{10}_+,
$$
there exists exactly one stable equilibrium point $E=E(K,\hat{p})$ which depends continuously on the data $(K,\hat{p})$. It is  natural to consider the transition dynamics of the stable equilibrium point $E(K,\hat{p})$ as a function of the carrying capacity $K$ keeping other parameters in $\hat{p}$ fixed. This transition dynamics has a clear biological meaning establishing the relationship between the infection transition rates $\hat{p}$  and the character of the corresponding equilibrium state. Furthermore, this also implies all possible scenarios how equilibrium states depend on the carrying capacity of the system. More precisely, in \cite{SKTW18a} there were only \textit{three} possible scenarios, each starting with the healthy equilibrium state for small values of $K$ and ending up at a certain equilibrium state when $K$ becomes sufficiently large and it was referred  as the \textit{continuity of the transition dynamics} of  stable equilibrium points.

It is naturally to expect that the continuity of the transition dynamics will hold  true for small positive $\beta_i$ and $\gamma_j$,  as a perturbation of \eqref{vanish}. The numerical simulations support this conjecture. It is, however, unreasonable to believe that   the latter property should  hold for any positive data $(\beta_1,\beta_2,\gamma_1,\gamma_2)$.
Below we prove a part of our conjecture.

It is convenient to  formulate our result for the relative carrying capacity $S^{**}$ instead of $K$ (note that $S^{**}$ differs from $K$ by a multiplicative constant only). Let us also introduce the threshold in \eqref{conditionG3}. Namely, let $\sigma_0$ denote the solution of the following equation:
\begin{equation}\label{conditionG30}
	\frac{b}{K\alpha_1}(\sigma_0-\sigma_1)= \min\left\{
\frac{\alpha_2(\sigma_2-\sigma_1)}{\gamma_1}, \,
\frac{\hat{\alpha}_3(\sigma_3-\sigma_1)}{\eta_1}
\right\}.
\end{equation}
Then
$$\sigma_0=\sigma_0(K)>\sigma_1
$$
and the global stability of $G_3$ holds whenever
$$
\sigma_1\le S^{**}\le \sigma_0.
$$
Combining the latter with Proposition~\ref{G.SG2} and Proposition~\ref{G.SG3}, we arrive at the following continuity of the transition dynamics.

\begin{theorem}
Let $\hat{p}>0$ be fixed. Then
\begin{itemize}
  \item[(i)] if $S^{**}$ is increasing and $0<S^{**}<\sigma_1$ then the only possible stable equilibrium state of \eqref{submodel1} is $G_2$;
  \item[(i')] when $S^{**}=\sigma_1$, $G_2$  coincides with $G_3$ and it is the only possible stable equilibrium state of \eqref{submodel1};
  \item[(ii)] there exists  $\sigma_0>\sigma_1$ such that $G_3$ is the only possible stable equilibrium state of \eqref{submodel1} for $\sigma_1<S^{**}<\sigma_0$
\end{itemize}
\end{theorem}

\begin{proof}
The claims immediately follow from the global stability  of the corresponding equilibria (indeed, by virtue of the global stability there can exist at most one locally stable point!)
\end{proof}

The case $S^{**}>\sigma_0$ is more subtle, but we have at least some local information near  $\hat\sigma=\sigma_1+\frac{K\alpha_1}{b}\Lambda$ and it follows from remark  \ref{R3} , $\hat\sigma>\sigma_0.$

The following theorem describes the bifurcation of equilibrium point $G_3$ in the case when  $\beta_i, \gamma_i, ~i=1,2$  are small  and $S^{**}$ cross the threshold $\hat\sigma.$
\begin{theorem}
Let $\hat{p}>0$ be fixed and $\beta_1,\beta_2,\gamma_1$and $\gamma_2$ are small then
   there exists $\epsilon>0$ such that if $S^* \in (\hat\sigma,\hat\sigma-\epsilon)$ then the equilbrium point  $G_3$ is unstable and the equilibrium point  $G_5$ is  stable and located near $G_3$.
\end{theorem}

\begin{proof}
We hive a sketch of the proof. Let
\begin{equation}\label{K1}
D:=\det B=(\alpha_2(\sigma_2-\sigma_1)+\gamma_2I_1^{*})(\mu_3-\alpha_3\sigma_1-\eta_1I_1^{*})-\beta_2(\gamma_1+\gamma_2)\sigma_1I_1^*,
\end{equation}
where we keep the notation
$$
I_1^*=\frac{b}{K\alpha_1}(S^{**}-\sigma_1).
$$
Let us consider the evolution of the equilibrium point $G_3$ for small values of the determinant  $D$ under an additional natural condition that $S^{**}>\sigma_1$. Certainly, the equilibrium point $G_3$ exists for any such admissible values. When $D>0$ is small, by Proposition~\ref{pro:G3} $G_3$ is locally stable, and when $D=0$ it looses the local stability. It turns out that at this moment $G_3$ bifurcates into a pair of points: (a) it continues as $G_3$ for small negative values of $D$, and (b) it appears one more equilibrium point of type $G_5$ in a neighborhood of $G_3$. Indeed, to describe the latter, we will seek it by perturbation of $G_3$ in the form
$$
Y_0=\sigma_1+\xi_0,\;\; Y_1=I_1^*+\xi_1,
$$
where $\xi_0$, $\xi_1$ are small real parameters, such that
$$
\lim_{D\to 0}\xi_0=\lim_{D\to 0}\xi_1=0.
$$
and
$$
Y_2=\lambda Y_3,\qquad  Y_3\neq 0, \qquad \lim_{D\to 0}Y_3=0.
$$
Then the third equation in \eqref{equil} yields
$$
\lambda=\frac{\beta_2Y_0}{\mu_2+\gamma_2Y_1+\eta_2Y_3-\alpha_2Y_0}
$$
For $\beta_2=0$ we have $\lambda=0$ and for $\beta_2>0$ we find
$$
\lambda=\lambda^*+o(1),\;\;\lambda^*=\frac{\beta_2\sigma_1}{\mu_2+\gamma_2I_1^*-\alpha_2\sigma_1}.
$$
is positive. Here $o(1)$ denotes the rest term when $D\to 0$. From the fourth equation in \eqref{equil} we obtain by eliminating of $Y_2$
\begin{equation}\label{K2}
(\mu_3-\alpha_3Y_0-\eta_1Y_1-\lambda\eta_2Y_3)({\mu_2+\gamma_2Y_1+\eta_2Y_3-\alpha_2Y_0})-\beta_2(\gamma_1+\gamma_2)Y_0Y_1=0.
\end{equation}
Now let us express $\xi_0$ and $\xi_1$ through $Y_3$. From the first two equations in \eqref{equil}  (with vanishing derivatives) we get
\begin{equation}\label{K3}
\xi_0=\frac{1}{\alpha_1}\Big(\eta_1+\lambda^*\gamma_1-\beta_1\frac{\sigma_1}{I_1^*}\Big)Y_3+O(Y_3^2)
\end{equation}
and
\begin{equation}\label{K4}
\xi_1=-\frac{1}{\alpha_1}\Big(\lambda^*\alpha_2+\delta\Big)Y_3-\frac{b\xi_0}{\alpha_1K}+O(Y_3^2).
\end{equation}
Using the notation (\ref{K1}) and setting
$$
A=\mu_3-\alpha_3\sigma_1-\eta_1I_1^{*},\;\;B=\alpha_2(\sigma_2-\sigma_1)+\gamma_2I_1^{*},
$$
we can write (\ref{K2}) as
$$
D+A(\gamma_2\xi_1+\eta_2Y_3-\alpha_2\xi_0)-B(\alpha_3\xi_0+\eta_1\xi_1+\lambda\eta_2Y_3)-\beta_2(\gamma_1+\gamma_2)(\sigma_1\xi_1+I_1^*\xi_0)+O(Y_3^2).
$$
This yields
\begin{equation}\label{K5}
Y_3=cD+O(D^2),
\end{equation}
where the coefficient $c$ can be easily evaluated in the case $\beta_1=\beta_2=0$ and $\gamma_1=\gamma_2=0$. Indeed, in this case $\lambda=0$,
$$
A=O(D)\;\;\mbox{and}\;\;\alpha_3\xi_0+\eta_1\xi_1=-\frac{b\eta_1^2}{K\alpha_1^2}.
$$
This yields
$$
B\frac{b\eta_1^2}{K\alpha_1^2}Y_3=-D+O(D^2)
$$
hence we obtain (\ref{K5}) with
$$
c=-\frac{K\alpha_1^2}{Bb\eta_1^2}.
$$
It follows from \eqref{K5} that for small negative $D$, the perturbation of $G_3$ bifurcates in an equilibrium point of type $G_5$ with \textit{positive} coordinates. By the continuity argument the latter property still holds true for small $\beta_i$ and $\gamma_j$.
\end{proof}

\section{Discussion}
In this paper we designed an SIR model as an extension of the model presented in  \cite{SKTW18a} and observed the effect of density dependence population regulation on disease dynamics. The complete local and global stability analysis of two boundary equilibrium points revealed that for small carrying capacity the disease free equilibrium point is always stable so disease can not persist in small population but for relatively large carrying capacity under some conditions we have one globally stable endemic equilibrium point. The existence of an endemic equilibrium point guarantees the persistence of the disease with a possible future threat of any outbreak in the population.
In future, similar to \cite{SKTW18a}, we would like to investigate the existence and uniqueness of the equilibrium point $G_5$. Since, the equilibrium point $G_3$ looses its stability and bifurcates to coexistence equilibrium point so it is worthwhile to show the existence of coexistence equilibrium in that case to understand the complete dynamics of the disease. We would like to see the dynamical behavior of system for significantly large $K$ and the case when $K = \infty.$

\subsection*{Conflict of interest}

On behalf of all authors, the corresponding author states that there is no conflict of interest.
\subsection*{Acknowledgements}
A PhD scholarship to Samia Ghersheen from Higher Eduaction Commision, Pakistan through SBK Women University Quetta, Pakistan is gratefully acknowledeged.

\bibliographystyle{wileyNJD-AMA}

\newcommand{\sectionbreak}{\clearpage}

\section*{Author Biography}
\begin{biography}{\includegraphics[height=86pt]{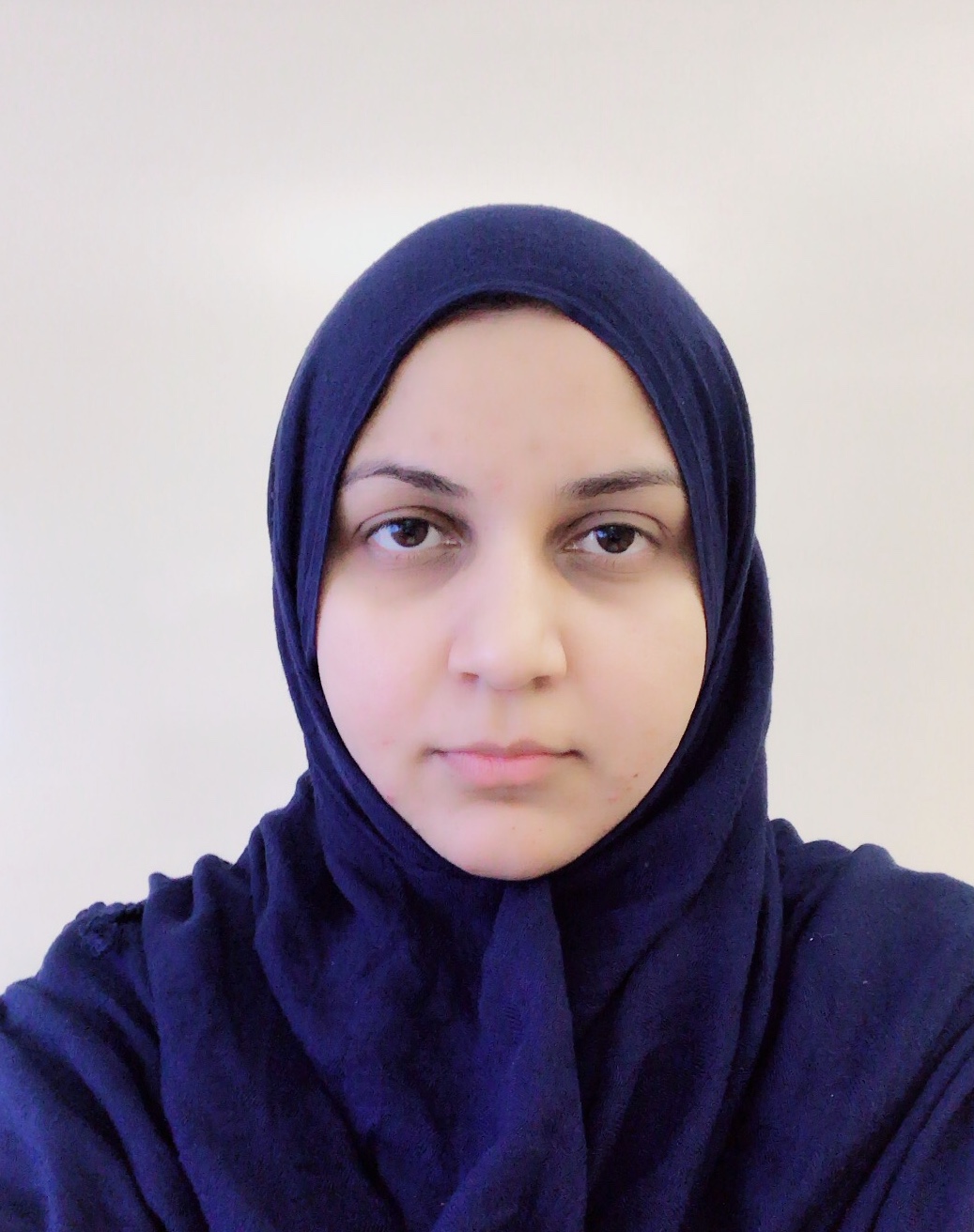}}	{\textbf{Samia Ghersheen}  has received the Master's degree in mathematics from Pakistan. She is currently PhD student with scholarship in the department of mathematics Link\"oping  University, Sweden. Her research area is mathematical modeling of infectious disease.}
\end{biography}
\bigskip
\bigskip
\begin{biography}{\includegraphics[height=86pt]{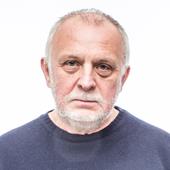}}{\textbf{ Vladimir Kozlov} graduated from the Leningrad University in 1976. He defended his PhD thesis in 1980 at the same University.
		He is a Doctor of Sciences in Mathematics  from 1990. In 1992 he moved to Link\"oping University, Sweden, and from 2010 he is a head of the division of Mathematics and Applied Mathematics at the same University. Vladimir Kozlov is an author of five books  and more than 160 articles. Research interests include applied mathematics (in particular, applications to theory elasticity, fluid mechanics, biology, medicine), partial differential equations, spectral theory of operators, asymptotic methods and inverse problems.}
\end{biography}
\bigskip
\bigskip
\begin{biography}{\includegraphics[height=86pt]{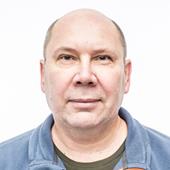}}{\textbf{ Vladimir Tkachev}  graduated from the Volgograd University in 1985. He defended his PhD thesis in 1990 at the Sobolev Institution of Mathematics.
		He is a Doctor of Sciences in Mathematics  from 1998. Since 2012 Vladimir Tkachev works  as Associate Professor at Link\"oping  University, Sweden. His research interests include partial differential equations (minimal surface equation), complex analysis and nonassociative algebras.}
\end{biography}
\bigskip
\bigskip
\begin{biography}{\includegraphics[height=86pt]{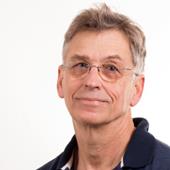}}{\textbf{Uno Wennergren}   is a professor in theoretical biology at Link\"oping University. He took his PhD 1994 by combining mathematics and biology. He did his postdoc by Peter Kareiva at University of Washington 1994 and 1995. Professor Wennergren is the head of a division of  15 researchers in theoretical biology. His own research spans a wide variety of mathematical modelling of systems and process of biological relevance: from animal welfare to intricate populations dynamics of ecological systems.}
\end{biography}




\end{document}